\documentclass[a4paper,10pt]{amsart}
\usepackage[utf8]{inputenc}
\usepackage{amsxtra}
\usepackage{amsopn}
\usepackage{amsmath,amsthm,amssymb}
\usepackage{amscd}
\usepackage{mathtools}
\usepackage{amsfonts}
\usepackage{latexsym}
\usepackage{verbatim}
\usepackage{MnSymbol}
\usepackage{xcolor}
\usepackage{cases}
\usepackage{tikz-cd}
\newcommand{\lr}{\longrightarrow}

\let\c\overline
\theoremstyle{plain}
\newtheorem{theorem}{Theorem}[section]
\newtheorem*{theorem*}{Theorem}

\newtheorem{proposition}[theorem]{Proposition}
\newtheorem{corollary}[theorem]{Corollary}
\newtheorem{remark}[theorem]{Remark}

\newtheorem*{mt*}{Main Theorem}
\newcommand\C{{\mathbb C}}
\newcommand\K{{\mathbb K}}

\newcommand\R{{\mathbb R}}
\newcommand\Z{{\mathbb Z}}

\newcommand{\de}[2]{\frac{\partial #1}{\partial #2}}

\newcommand{\del}{\partial}
\newcommand{\delbar}{\overline{\del}}

\newcommand{\cinf}{\mathcal{C}^\infty}
\renewcommand{\H}{\mathcal{H}}

\DeclareMathOperator{\vol}{Vol}

\DeclareMathOperator{\real}{Re}
\DeclareMathOperator{\img}{Im}

\DeclareMathOperator{\im}{im}
\DeclareMathOperator{\GL}{GL}
\DeclareMathOperator{\End}{End}
\DeclareMathOperator{\id}{id}
\DeclareMathOperator{\rank}{rank}
\DeclareMathOperator{\Hom}{Hom}
\DeclareMathOperator{\lv}{\lvert}
\DeclareMathOperator{\rv}{\rvert}
\let\c\overline

\title{Bott-Chern Laplacian on almost Hermitian manifolds}
\author{Riccardo Piovani}
\address{Dipartimento di Matematica\\
Universit\`{a} di Pisa\\
Largo Bruno Pontecorvo 5 \\
56127 Pisa, Italy}
\address{Current Address: Dipartimento di Scienze Matematiche, Fisiche e Informatiche\\
Unit\`{a} di Matematica e Informatica\\
Universit\`{a} degli Studi di Parma\\
Parco Area delle Scienze 53/A \\
43124 Parma, Italy}
\email{riccardo.piovani@unipr.it}

\author{Adriano Tomassini}
\address{Dipartimento di Scienze Matematiche, Fisiche e Informatiche\\
Unit\`{a} di Matematica e Informatica\\
Universit\`{a} degli Studi di Parma\\
Parco Area delle Scienze 53/A \\
43124 Parma, Italy}
\email{adriano.tomassini@unipr.it}

\thanks{\newline 
The first author is partially supported by GNSAGA of INdAM.
The second author is partially supported by the Project PRIN 2017 ``Real and Complex Manifolds: Topology, Geometry and holomorphic dynamics'' and by GNSAGA of INdAM}

\begin{document}
\maketitle

\begin{abstract} 
Let $(M,J,g,\omega)$ be a $2n$-dimensional almost Hermitian manifold. We extend the definition of the Bott-Chern Laplacian on $(M,J,g,\omega)$, proving that it is still elliptic.
On a compact K\"ahler manifold, the kernels of the Dolbeault Laplacian and of the Bott-Chern Laplacian coincide.  We show that such a property does not hold when $(M,J,g,\omega)$ is a compact almost K\"ahler manifold, providing an explicit almost K\"ahler structure on the Kodaira-Thurston manifold. 
Furthermore, if $(M,J,g,\omega)$ is a connected compact almost Hermitian $4$-manifold, denoting by $h^{1,1}_{BC}$ the dimension of the space of Bott-Chern harmonic $(1,1)$-forms, we prove that either $h^{1,1}_{BC}=b^-$ or $h^{1,1}_{BC}=b^-+1$. In particular, if $g$ is almost K\"ahler, then $h^{1,1}_{BC}=b^-+1$, extending the result by Holt and Zhang \cite{HZ} for the kernel of the Dolbeault Laplacian. We also show that the dimensions of the spaces of Bott-Chern and Dolbeault harmonic $(1,1)$-forms behave differently on almost complex 4-manifolds endowed with strictly locally conformally almost K\"ahler metrics. Finally, we relate some spaces of Bott-Chern harmonic forms to the Bott-Chern cohomology groups for almost complex manifolds, recently introduced in \cite{CPS}.

%

\end{abstract}

\section{Introduction}
On a complex manifold, given a Hermitian metric, several elliptic operators naturally arise from the union of the complex and the Hermitian structure. As a typical example, the {\em Dolbeault Laplacian} is defined as $\Delta_{\delbar}=\delbar\delbar^*+\delbar^*\delbar$, where the exterior differential defined on the space $A^{p,q}$ of $(p,q)$-forms decomposes as $d=\del+\delbar$ and, if $*:A^{p,q}\lr A^{n-q,n-p}$ is the $\C$-linear complex Hodge star operator, where $n$ is the complex dimension of the complex manifold, then $\del^*=-*\delbar*$ and $\delbar^*=-*\del*$ are the formal adjoints of the operators $\del$ and $\delbar$, respectively. Denote by $\H^{p,q}_{\delbar}$ the space of Dolbeault harmonic forms, i.e., the kernel of $\Delta_{\delbar}$. Since $\Delta_{\delbar}$ is elliptic, when the manifold is compact, by Hodge theory then $\H^{p,q}_{\delbar}$ is isomorphic to the Dolbeault cohomology 
\[
H^{p,q}_{\delbar}=\frac{\ker \delbar}{\im\delbar}
\]
and it is finite dimensional; denote by $h^{p,q}_{\delbar}$ its finite complex dimension.

In 1960 Kodaira and Spencer, \cite{KS}, in order to prove the stability of the K\"ahler condition under small deformations,  introduced the following $4^{th}$-order elliptic and formally self adjoint differential operator
\begin{equation*}
\tilde\Delta_{BC} \;=\;
\del\delbar\delbar^*\del^*+
\delbar^*\del^*\del\delbar+\del^*\delbar\delbar^*\del+\delbar^*\del\del^*\delbar
+\del^*\del+\delbar^*\delbar.
\end{equation*}
Schweitzer, in 2007, \cite{S}, studied the operator $\tilde\Delta_{BC}$ on compact Hermitian manifolds, naming it the {\em Bott-Chern Laplacian}. In particular, denoting by $\H^{p,q}_{BC}$ the space of Bott-Chern harmonic $(p,q)$-forms on a given compact Hermitian manifold $(M,J,g,\omega)$, he proved the following Bott-Chern decomposition of the space of $(p,q)$-forms
\begin{gather}\label{bc-decomp}
A^{p,q}=\H^{p,q}_{BC}\overset{\perp}{\oplus}{\del\delbar A^{p-1,q-1}}\overset{\perp}{\oplus}{\del^* A^{p+1,q}+\delbar^* A^{p,q+1}}.
\end{gather}
As a consequence, the space $\H^{p,q}_{BC}$ is finite dimensional and $\H^{p,q}_{BC}\cong H^{p,q}_{BC}$, where
\[
H^{p,q}_{BC}=\frac{\ker d}{\im \del\delbar}
\]
denotes the $(p,q)$-{\em Bott-Chern cohomology group}.
In particular, the complex dimension $h^{p,q}_{BC}=\dim_\C\H^{p,q}_{BC}$ is a complex invariant of $(M,J)$, which does not depend on the Hermitian metric $g$.

If the compact Hermitian manifold $(M,J,g,\omega)$ is K\"ahler, i.e., $d\omega=0$, then
\begin{equation}\label{bc-lapl-kahl}
\tilde\Delta_{BC} =\Delta_{\delbar}\Delta_{\delbar}
+\del^*\del+\delbar^*\delbar
\end{equation} and
the spaces of Bott-Chern and Dolbeault harmonic forms coincide, i.e.,
\begin{equation}\label{eq-bc-delbar}
\H^{p,q}_{BC}=\H^{p,q}_{\delbar}.
\end{equation}

Now, let $(M,J,g,\omega)$ be an almost Hermitian manifold, i.e., the almost complex structure $J$ may not be integrable, i.e., $J$ may not derive from a complex-manifold structure on $M$. The exterior differential decomposes as $d=\mu+\del+\delbar+\c\mu$, and Dolbeault and Bott-Chern cohomologies are, in general, no more well defined. However, the Dolbeault Laplacian $\Delta_{\delbar}$ is still well defined and elliptic, resulting in $\H^{p,q}_{\delbar}$ being finite dimensional when $M$ is compact.

The study of Dolbeault harmonic forms on almost Hermitian manifold of real dimension 4 has been very recentely developed by Holt and Zhang, \cite{HZ, HZ2}, and by Tardini and the second author, \cite{TT}. Holt and Zhang, working on the Kodaira-Thurston manifold, showed that the number $h^{0,1}_{\delbar}$ may become arbitrarily large when varying continuously the almost complex structure with an associated almost K\"ahler metric and that $h^{0,1}_{\delbar}$ may vary with different choices of almost Hermitian metrics. Furthermore, they proved that $h^{0,1}_{\delbar}$ may vary with different choices of almost K\"ahler metrics. In this way they answered a question by Kodaira and Spencer, \cite[Problem 20]{Hi}. Moreover, they showed $h^{1,1}_{\delbar}=b^-+1$ on every compact almost K\"ahler 4-manifold, where $b^-$ is the dimension of the space of anti-self-dual, i.e., $*\alpha=-\alpha$, Hodge harmonic 2-forms, which is a topological invariant. Tardini and the second author proved that $h^{1,1}_{\delbar}=b^-$ on every compact almost complex 4-manifold with a strictly locally conformally almost K\"ahler metric.

In this paper, we focus on the study of the Bott-Chern Laplacian on almost Hermitian manifolds.
Note that, analogously to the Dolbeault Laplacian, also the Bott-Chern Laplacian $\tilde\Delta_{BC}$ is still well defined on almost Hermitian manifolds $(M,J,g,\omega)$, and it is straightforward to show that it is also elliptic, see Proposition \ref{prop-ell}. Therefore, when $M$ is compact, the Bott-Chern decomposition \eqref{bc-decomp} still holds, and $\H^{p,q}_{BC}$ is finite dimensional. 

We prove the following
\begin{theorem*}[Theorem \ref{thm-11}] \label{main-1}
Let $(M,J,g,\omega)$ be a compact almost Hermitian manifold of real dimension $4$. Then either $h^{1,1}_{BC}=b^-$ or $h^{1,1}_{BC}=b^-+1$.
\end{theorem*}
Moreover, we specialize the previous theorem when the almost Hermitian metric $\omega$ is almost K\"ahler, i.e., $d\omega=0$, obtaining that $h^{1,1}_{BC}$ is independent of the choice of an almost K\"ahler metric on a given compact almost complex 4-manifold, that is, 
\begin{theorem*}[Corollary \ref{cor-11}]\label{main-2}
Let $(M,J,g,\omega)$ be a compact almost K\"ahler manifold of real dimension $4$.
Then, $h^{1,1}_{BC}=b^-+1$ and $\H^{1,1}_{BC}=\H^{1,1}_{\delbar}$.
\end{theorem*}
Note that in the integrable case, i.e., on compact complex surfaces, it holds that $h^{1,1}_{BC}=b^-+1$ on K\"ahler surfaces, on complex surfaces diffeomorphic to solvmanifolds, and on complex surfaces of class VII (see \cite[Chapter IV, Theorem 2.7]{BPV} and \cite{ADT}).

We also provide a non integrable almost complex structure on a hyperelliptic surface, endowed with a strictly locally conformally almost K\"ahler metric, such that $h^{1,1}_{BC}=b^-+1$. This proves that the dimension of Bott-Chern harmonic $(1,1)$-forms behaves differently than the dimension of Dolbeault harmonic $(1,1)$-forms, \cite{TT}, when the almost complex 4-manifold is endowed with a strictly locally conformally almost K\"ahler metric.

Very recently Holt improved the result of Theorem \ref{thm-11}, by showing that $$h^{1,1}_{BC}=b^-+1$$ on any given compact almost Hermitian $4$-manifold, see \cite[Theorem 4.2]{Ho}.

Taking into account the integrable case, one may ask whether \eqref{bc-lapl-kahl} and \eqref{eq-bc-delbar} holds or not, when the almost Hermitian metric is almost K\"ahler. We show that \eqref{eq-bc-delbar} is not true, describing an explicit example on the Kodaira-Thurston manifold. This also implies that \eqref{bc-lapl-kahl} does not hold.
In fact, working on a family of almost K\"ahler metrics on the Kodaira-Thurston manifold, we show
\begin{theorem*}[Corollary \ref{cor-kod-2}]\label{main-3}
There exists an almost K\"ahler $4$-manifold $(M,J,g,\omega)$ such that for some bidegree $(p,q)$ it holds that
\begin{equation*}
\H^{p,q}_{BC}\ne\H^{p,q}_{\delbar}.
\end{equation*}
\end{theorem*}

Finally, we recall a very recent definition of Bott-Chern cohomology for almost complex manifolds, \cite{CPS}, obtaining a natural injection of some spaces of Bott-Chern harmonic forms into this new Bott-Chern cohomology.

For other results concerning Bott-Chern-like harmonic forms on almost complex manifolds, equipped with cohomological counterparts, see \cite{TT0}.\smallskip

The present paper is organized in the following way. 
In section \ref{preliminaries}, we review some basic facts on almost complex manifolds and elliptic differential operators. 
Section \ref{bc-aeppli} is devoted to the definition and to the proof of the fundamental properties of the Bott-Chern Laplacian in the almost complex setting.
In section \ref{sec-h11}, we study Bott-Chern harmonic $(1,1)$-forms on almost Hermitian 4-manifolds, proving Theorem \ref{thm-11}.
In section \ref{bc-harmonic-kodaira}, we describe a family of almost K\"ahler structures on the Kodaira-Thurston manifold, comparing the two spaces $\H^{p,q}_{BC}$ and $\H^{p,q}_{\delbar}$ and proving Corollary \ref{cor-kod-2}. 
In section \ref{sec-lck}, we describe an almost complex structure on a hyperelliptic surface, endowed with a strictly locally conformally almost K\"ahler metric, such that $h^{1,1}_{BC}=b^-+1$.
Finally, in section \ref{cohomology}, we recall a very recent definition of Bott-Chern cohomology for almost complex manifolds by Coelho, Placini and Stelzig \cite{CPS}, and briefly analyse its relation with the space of Bott-Chern harmonic forms (see Proposition \ref{BC-cohomology}).

\medskip\medskip
\noindent{\em Acknowledgments.}
We are sincerely grateful to Tom Holt, Nicoletta Tardini and Weiyi Zhang for interesting comments and useful conversations on the subject of the paper. In particular, we would like to thank Tom Holt for having suggested to us that the characterization \eqref{eq-11-bc} of the space of Bott-Chern harmonic $(1,1)$-forms $\H^{1,1}_{BC}$, appearing in the first part of the proof of Theorem \ref{thm-11}, actually yields that the dimension $h^{1,1}_{BC}$ of $\H^{1,1}_{BC}$ can be only equal to either $b^-+1$ or $b^-$. 

The authors want also to express their gratitude to the anonymous Referee for his/her useful suggestions, which led to a better presentation of the results presented in the paper.

\section{Preliminaries}\label{preliminaries}
Throughout this paper, we will only consider connected manifolds without boundary.

Let $(M,J)$ be an almost complex manifold of dimension $2n$, i.e., a $2n$-differentiable manifold together with an almost complex structure $J$, that is $J\in\End(TM)$ and $J^2=-\id$. The complexified tangent bundle $T_{\C}M=TM\otimes\C$ decomposes into the two eigenspaces of $J$ associated to the eigenvalues $i,-i$, which we denote respectively by $T^{1,0}M$ and $T^{0,1}M$, giving
\begin{equation*}
T_{\C}M=T^{1,0}M\oplus T^{0,1}M.
\end{equation*}
Denoting by $\Lambda^{1,0}M$ and $\Lambda^{0,1}M$ the dual vector bundles of $T^{1,0}M$ and $T^{0,1}M$, respectively, we set
\begin{equation*}
\Lambda^{p,q}M=\bigwedge^p\Lambda^{1,0}M\wedge\bigwedge^q\Lambda^{0,1}M
\end{equation*}
to be the vector bundle of $(p,q)$-forms, and let $A^{p,q}=A^{p,q}(M)=\Gamma(\Lambda^{p,q}M)$ be the space of smooth sections of $\Lambda^{p,q}M$. We denote by $A^k=A^k(M)=\Gamma(\Lambda^{k}M)$ the space of $k$-forms. Note that $\Lambda^{k}M\otimes\C=\bigoplus_{p+q=k}\Lambda^{p,q}M$.

Let $f\in\cinf(M,\C)$ be a smooth function on $M$ with complex values. Its differential $df$ is contained in $A^1\otimes\C=A^{1,0}\oplus A^{0,1}$. On complex 1-forms, the exterior differential acts as
\[
d:A^1\otimes\C\to A^2\otimes\C=A^{2,0}\oplus A^{1,1}\oplus A^{0,2}.
\]
 Therefore, it turns out that the differential operates on $(p,q)$-forms as
\begin{equation*}
d:A^{p,q}\to A^{p+2,q-1}\oplus A^{p+1,q}\oplus A^{p,q+1}\oplus A^{p-1,q+2},
\end{equation*}
where we denote the four components of $d$ by
\begin{equation*}
d=\mu+\del+\delbar+\c\mu.
\end{equation*}
From the relation $d^2=0$, we derive
\begin{equation*}
\begin{cases}
\mu^2=0,\\
\mu\del+\del\mu=0,\\
\del^2+\mu\delbar+\delbar\mu=0,\\
\del\delbar+\delbar\del+\mu\c\mu+\c\mu\mu=0,\\
\delbar^2+\c\mu\del+\del\c\mu=0,\\
\c\mu\delbar+\delbar\c\mu=0,\\
\c\mu^2=0.
\end{cases}
\end{equation*}

Let $(M,J)$ be an almost complex manifold. If the almost complex structure $J$ is induced from a complex manifold structure on $M$, then $J$ is called integrable. It is equivalent to the decomposition of the exterior differential as $d=\del+\delbar$.

A Riemannian metric on $M$ for which $J$ is an isometry is called almost Hermitian.
Let $g$ be an almost Hermitian metric, the $2$-form $\omega$ such that
\begin{equation*}
\omega(u,v)=g(Ju,v)\ \ \forall u,v\in\Gamma(TM)
\end{equation*}
is called the fundamental form of $g$. We will call $(M,J,g,\omega)$ an almost Hermitian manifold.
We denote by $h$ the Hermitian extension of $g$ on the complexified tangent bundle $T_\C M$, and by the same symbol $g$ the $\C$-bilinear symmetric extension of $g$ on $T_\C M$. Also denote by the same symbol $\omega$ the $\C$-bilinear extension of the fundamental form $\omega$ of $g$ on $T_\C M$. 
Thanks to the elementary properties of the two extensions $h$ and $g$, we may want to consider $h$ as a Hermitian operator $T^{1,0}M\times T^{1,0}M\to\C$ and $g$ as a $\C$-bilinear operator $T^{1,0}M\times T^{0,1}M\to\C$.
Recall that $h(u,v)=g(u,\bar{v})$ for all $u,v\in \Gamma(T^{1,0}M)$.

Let $(M,J,g,\omega)$ be an almost Hermitian manifold of real dimension $2n$. Extend $h$ on $(p,q)$-forms and denote the Hermitian inner product by $\langle\cdot,\cdot\rangle$.
Let $*:A^{p,q}(M)\lr A^{n-q,n-p}(M)$ the $\C$-linear extension of the standard Hodge $*$ operator on Riemannian manifolds with respect to the volume form $\vol=\frac{\omega^n}{n!}$, i.e., $*$ is defined by the relation
\[
\alpha\wedge{*\c\beta}=\langle\alpha,\beta\rangle\vol\ \ \ \forall\alpha,\beta\in A^{p,q}.
\]
Then the operators
\begin{equation*}
d^*=-*d*,\ \ \ \mu^*=-*\c\mu*,\ \ \ \del^*=-*\delbar*,\ \ \ \delbar^*=-*\del*,\ \ \ \c\mu^*=-*\mu*,
\end{equation*}
are the formal adjoint operators respectively of $d,\mu,\del,\delbar,\c\mu$. Recall $\Delta_{d}=dd^*+d^*d$ is the Hodge Laplacian, and, as in the integrable case, set 
\begin{equation*}
\Delta_{\del}=\del\del^*+\del^*\del,\ \ \ \Delta_{\delbar}=\delbar\delbar^*+\delbar^*\delbar,
\end{equation*}
respectively as the $\del$ and $\delbar$ Laplacians.

If $M$ is compact, then we easily deduce the following relations
\begin{equation*}
\begin{cases}
\Delta_{d}\alpha=0\ &\iff\ d\alpha=0,\ d*\alpha=0,\ \forall\alpha\in A^k\\
\Delta_{\del}\alpha=0\ &\iff\ \del\alpha=0,\ \delbar*\alpha=0,\ \forall\alpha\in A^{p,q}\\
\Delta_{\delbar}\alpha=0\ &\iff\ \delbar\alpha=0,\ \del*\alpha=0,\ \forall\alpha\in A^{p,q}
\end{cases}
\end{equation*}
which characterizes the spaces of harmonic forms
\begin{equation*}
\H^{k}_{d},\ \ \ \H^{p,q}_{\del},\ \ \ \H^{p,q}_{\delbar},
\end{equation*}
defined as the spaces of forms which are in the kernel of the associated Laplacians.
All these Laplacians are elliptic operators on the almost Hermitian manifold $(M,J,g,\omega)$, implying that all the spaces of harmonic forms are finite dimensional when the manifold is compact. Denote by
\begin{equation*}
b^{k},\ \ \ h^{p,q}_{\del},\ \ \ h^{p,q}_{\delbar}
\end{equation*}
respectively the real dimension of $\H^k_d$ and the complex dimensions of $\H^{p,q}_{\del},\H^{p,q}_{\delbar}$.

Since we will need to use the maximum principle for second order uniformly elliptic differential operators, let us recall some definitions and the results which will be useful.
Let $M$ be a differentiable manifold of dimension $m$, and let $E,F$ be $\K$-vector bundles over $M$, with $\K=\R$ or $\K=\C$, $\rank E=r$, $\rank F=s$.
A differential operator of order $l$ from $E$ to $F$ is a $\K$-linear operator $P:\Gamma(M,E)\to \Gamma(M,F)$ of the form
\begin{equation*}
Pu(x)=\sum_{\lv\alpha\rv\le l}a_\alpha(x)D^\alpha u(x)\ \ \ \forall x\in\Omega,
\end{equation*}
where $E_{|\Omega}\simeq\Omega\times\K^r$, $F_{|\Omega}\simeq\Omega\times\K^s$ are trivialized locally on some open chart $\Omega\subset M$ equipped with local coordinates $x^1,\dots,x^{m}$, and the functions 
\begin{equation*}
a_\alpha(x)=(a_{\alpha ij}(x))_{1\le i\le s,1\le j\le r}
\end{equation*}
are $s\times r$ matrices with smooth coefficients on $\Omega$. Here $D^\alpha=(\del/\del x^1)^{\alpha_1}\dots(\del/\del x^{m})^{\alpha_{m}}$, and $u=(u_j)_{1\le j\le r}$, $D^\alpha u=(D^\alpha u_j)_{1\le j\le r}$ are viewed as column matrices. Moreover, we require $a_\alpha\nequiv0$ for some open chart $\Omega\subset M$ and for some $\lv\alpha\rv= l$.

Let $P:\Gamma(M,E)\to \Gamma(M,F)$ be a $\K$-linear differential operator of order $l$ from $E$ to $F$. We define the principal symbol of $P$ as the operator
\begin{equation*}
\sigma_P:T^*M\to \Hom(E,F)\ \ \ (x,\xi)\mapsto\sum_{\lv\alpha\rv= l}a_\alpha(x)\xi^\alpha,
\end{equation*}
where $\xi^\alpha=(\xi_1)^{\alpha_1}\dots(\xi_m)^{\alpha_m}$, given that $\xi=(\xi_1,\dots,\xi_m)$.
Note that, if $u\in\Gamma(M,E)$ is a smooth section of $E$ and $f\in\cinf(M)$ is a smooth real valued function, then
\begin{equation}\label{eq-sym}
P(f^lu)(x)=l!\sigma_P(x,df(x))(u(x)).
\end{equation}
We say that $P$ is elliptic if $\sigma_P(x,\xi)\in\Hom(E_x,F_x)$ is an isomorphism for every $x\in M$ and $0\ne\xi\in T_x^*M$. By (\ref{eq-sym}), we observe that $P$ is elliptic if and only if for all $x\in M$, $u\in\Gamma(M,E)$ and $f\in\cinf(M)$ such that $u(x)\ne0$, $f(x)=0$ and $df(x)\ne0$ we have
\begin{equation*}
P(f^lu)(x)\ne0.
\end{equation*}
Let $E=F$ and consider a Riemannian or a Hermitian metric $g$ on $E$.
We say that $P$ is strongly elliptic if $l=2k$ and there exists $C>0$ such that
\begin{equation*}
(-1)^k\real( g(\sigma_P(x,\xi)(u(x)),u(x))) \ge C\lv\xi\rv^{2k}g(u(x),u(x))
\end{equation*}
for all $x\in M$, $u\in\Gamma(M,E)$ and $\xi\in T^*M$, see \cite[Definition 4.2]{F}.

We will make use of the following statement of the maximum principle for strongly elliptic operators of second order, see \cite[Chapter 6, Section 4, Theorem 3]{E}.
\begin{theorem}
Let $\Omega\subset M$ be a relatively compact domain, with $\c\Omega$ contained in a local chart, and let $P:\cinf(\Omega)\to\cinf(\Omega)$ be a strongly elliptic $\R$-linear differential operator of order $2$ without zero order terms, i.e., such that $P(1)=0$. If $Pu=0$ in $\Omega$ and $u\in\mathcal{C}(\c\Omega)$ attains its maximum or minimum over $\c\Omega$ at an interior point, then $u$ is constant within $\Omega$.
\end{theorem}

\section{Bott-Chern and Aeppli Laplacians}\label{bc-aeppli}

Let $(M,J,g,\omega)$ be an almost Hermitian manifold. As in the integrable setting, we define
\begin{equation*}
\tilde\Delta_{BC}=
\del\delbar\delbar^*\del^*+
\delbar^*\del^*\del\delbar+\del^*\delbar\delbar^*\del+\delbar^*\del\del^*\delbar
+\del^*\del+\delbar^*\delbar,
\end{equation*}
and
\begin{equation*}
 \tilde\Delta_{A}= \del\delbar\delbar^*\del^*+
\delbar^*\del^*\del\delbar+
\del\delbar^*\delbar\del^*+\delbar\del^*\del\delbar^*+
\del\del^*+\delbar\delbar^*,
\end{equation*}
and still call them the Bott-Chern and the Aeppli Laplacian, respectively.
Note that
\begin{equation}\label{bc-a-duality}
*\tilde\Delta_{BC}=\tilde\Delta_{A}*\ \ \ \tilde\Delta_{BC}*=*\tilde\Delta_{A}.
\end{equation}

If $M$ is compact, then we easily deduce the following relations
\begin{equation*}
\begin{cases}
\tilde\Delta_{BC}\alpha=0\ \iff\  \del\alpha=0,\ \delbar\alpha=0,\ \del\delbar*\alpha=0,\ \forall\alpha\in A^{p,q},\\
\tilde\Delta_{A}\alpha=0\ \iff\  \del*\alpha=0,\ \delbar*\alpha=0,\ \del\delbar\alpha=0,\ \forall\alpha\in A^{p,q},
\end{cases}
\end{equation*}
which characterize the spaces of harmonic $(p,q)$-forms
\begin{equation*}
\H^{p,q}_{BC},\ \ \ \H^{p,q}_{A},
\end{equation*}
defined as the spaces of $(p,q)$-forms which are in the kernel of the associated Laplacians.
\begin{remark}
By equation \eqref{bc-a-duality}, note that $*\H^{p,q}_{BC}=\H^{n-q,n-p}_{A}$ and $*\H^{p,q}_{A}=\H^{n-q,n-p}_{BC}$. In the following, we will study only the spaces $\H^{p,q}_{BC}$ on an almost complex manifolds; this is sufficient to describe also the spaces $\H^{p,q}_{A}$.
\end{remark}

We are interested in studying the kernel of the Bott-Chern Laplacian $\tilde\Delta_{BC}$ on almost complex manifolds. The kernel of an elliptic operator is finite dimensional on a compact manifold. Therefore, the first thing we verify is that $\tilde\Delta_{BC}$ is elliptic. The proofs known by the authors of the ellipticity of $\tilde\Delta_{BC}$, see, e.g., \cite[Proposition 5]{KS} by Kodaira and Spencer or \cite[Page 8]{S} by Schweitzer, make use of local complex coordinates to compute explicitly the symbol of $\tilde\Delta_{BC}$, therefore do not hold anymore on almost complex manifolds. Nonetheless, these proofs could be adapted to compute the symbol in suitable local frames on almost complex manifolds.
\begin{proposition}\label{prop-ell}
Let $(M,g,J,\omega)$ be an almost Hermitian manifold of real dimension $2n$. The Bott-Chern Laplacian $\tilde\Delta_{BC}$ is elliptic.
\end{proposition}
\begin{proof}
To compute the symbol of $\tilde\Delta_{BC}$, choose a local coframe $\{\theta^1,\dots,\theta^n\}$ on $A^{1,0}$ such that the almost Hermitian metric is written 
\[
\omega=i\sum_{k=1}^n\theta^{k\c{k}}.
\]
We write a form $\alpha\in A^{p,q}$ locally as
\[
\alpha=\alpha_{i_1\dots i_pj_1\dots j_q}\theta^{i_1}\wedge\dots\wedge\theta^{\c{j_q}}.
\]
Its differential then acts as
\begin{align*}
d\alpha&=d\alpha_{i_1\dots j_q}\wedge\theta^{i_1}\wedge\dots\wedge\theta^{\c{j_q}}+\alpha_{i_1\dots j_q}d(\theta^{i_1}\wedge\dots\wedge\theta^{\c{j_q}})\\
&=\del\alpha_{i_1\dots j_q}\wedge\theta^{i_1}\wedge\dots\wedge\theta^{\c{j_q}}+\delbar\alpha_{i_1\dots j_q}\wedge\theta^{i_1}\wedge\dots\wedge\theta^{\c{j_q}}+\alpha_{i_1\dots j_q}d(\theta^{i_1}\wedge\dots\wedge\theta^{\c{j_q}}).
\end{align*}
In calculating the symbol, we are only interested in the highest order derivatives acting on $\alpha_{i_1\dots j_q}$.
Therefore, for the purpose of computing the symbol, we note that $\del$ and $\delbar$ behave like on a complex manifold. The same reasoning works for $\del^*$ and $\delbar^*$. Since $\tilde\Delta_{BC}$ is  elliptic on complex manifolds, this ends the proof.
\end{proof}
The same considerations in the proof of Proposition \ref{prop-ell} also prove that the $\del$, $\delbar$, and the Aeppli Laplacians are  elliptic, too.

Denote by
\begin{equation*}
h^{p,q}_{BC},\ \ \ h^{p,q}_{A}
\end{equation*}
respectively the finite complex dimensions of $\H^{p,q}_{BC}$ and of $\H^{p,q}_{A}$.

\section{Bott-Chern harmonic $(1,1)$-forms on almost Hermitian 4-manifolds}\label{sec-h11}

The goal of this section is to study the space of Bott-Chern harmonic forms of bidegree $(1,1)$ on almost Hermitian manifolds of real dimension $4$. We start noting that this space is a conformal invariant of the metric.

\begin{remark}\label{conformal}
Let $(M,J)$ be a compact almost complex manifold of real dimension $2n$.
Let $\tilde\omega$, $\omega=e^t\tilde\omega$, with $t\in\cinf(M)$, be two conformal almost Hermitian metrics. The two Hodge star operators behave, on the space $A^{p,q}$, as
\[
*_{\omega}=e^{t(n-p-q)}*_{\tilde\omega}.
\]
Therefore, when $p+q=n$, the space $\H^{p,q}_{BC}$ is a conformal invariant of almost Hermitian metrics, thanks to the characterization
\[
\tilde\Delta_{BC}\alpha=0\ \iff \del\alpha=0,\ \delbar\alpha=0,\ \del\delbar*\alpha=0,\ \ \ \forall\alpha\in A^{p,q}.
\]
In particular, $h^{p,q}_{BC}$ is also a conformal invariant of almost Hermitian metrics for $p+q=n$. 
This is especially true when $2n=4$ and $p=q=1$. 

By a remarkable result of Gauduchon, \cite{Ga}, for any given almost Hermitian metric $\tilde\omega$ on the compact almost complex $2n$-manifold $(M,J)$, there always exists a unique, up to homothety,  Gauduchon metric $\omega$ conformal to $\tilde\omega$, i.e., $\del\delbar\omega^{n-1}=0$. In particular, for $2n=4$, we have $\del\delbar\omega=0$.
\end{remark}

Let $(M,g)$ be a compact oriented Riemannian manifold of real dimension 4, and set 
$$\Lambda^-=\{\alpha\in \Lambda^2M\,:\,*\alpha=-\alpha\}
$$
the bundle of anti-self-dual $2$-forms. 
Denote by 
\[
\H^-=\{\alpha\in A^2\,:\,\Delta_d\alpha=0,\,*\alpha=-\alpha\},
\]
the subspace of harmonic anti-self-dual $2$-forms
and set $b^-=\dim_{\R}\H^-$.
Note that $b^-$ is metric independent: see \cite[Chapter 1]{DK} for its topological meaning.

\begin{remark}
Let $(M,J,g,\omega)$ be a compact almost Hermitian manifold of real dimension $4$. Note that the space of harmonic anti-self-dual complex valued $2$-forms $\H^-\otimes\C$ is indeed a subspace of $A^{1,1}$, which will be denoted by $\H^-_{\C}$.
We remark that every harmonic anti-self-dual $(1,1)$-form $\gamma\in \H^-_{\C}$, i.e., $d\gamma=0$ and $*\gamma=-\gamma$, is a Bott-Chern harmonic $(1,1)$-form. In fact, it holds $\del\gamma=0,$ $\delbar\gamma=0,$ $\del\delbar*\gamma=0$. Hence, $h^{1,1}_{BC}\ge b^-$.
\end{remark}

Now we can state and prove the following theorem, gaining a topological interpretation of the dimension $h^{1,1}_{BC}$.

\begin{theorem}\label{thm-11}
Let $(M,J,g,\omega)$ be a compact almost Hermitian manifold of real dimension $4$. Then either $h^{1,1}_{BC}=b^-$ or $h^{1,1}_{BC}=b^-+1$.
\end{theorem}
\begin{proof}
Since $h^{1,1}_{BC}$ is a conformal invariant of the metric, up to a conformal change of the Hermitian metric $g$, we can assume in this proof that $\omega$ is Gauduchon, i.e., $\del\delbar\omega=0$.

We divide the proof in two steps.

(I) First, we prove that the space of Bott-Chern harmonic $(1,1)$-forms is
\begin{equation}\label{eq-11-bc}
\H^{1,1}_{BC}=\{f\omega+\gamma\in A^{1,1}\,|\,f\in\C,\ *\gamma=-\gamma,\ d(f\omega+\gamma)=0\}.
\end{equation}

(II) Then, we prove that the complex dimension of $\H^{1,1}_{BC}$ can only be equal to either $b^-$ or $b^-+1$.

(I) From the Lefschetz decomposition for $2$-forms, see \cite[Proposition 1.2.30]{H}, one gets the following decomposition:
\begin{equation}\label{eq-dec11}
\Lambda^{1,1}M=\C(\omega)\oplus(\Lambda^-\otimes\C).
\end{equation}
Let $\phi\in\H^{1,1}_{BC}$. By equation (\ref{eq-dec11}), we have $\phi=f\omega+\gamma$, where $f$ is a smooth function with complex values on $M$ and $*\gamma=-\gamma$.
To prove the characterization of $\H^{1,1}_{BC}$, we claim that $f$ is a complex constant. Note that
\begin{align}
\del\phi=0&\ \iff\ \del f\wedge\omega+f\del\omega+\del\gamma=0,\label{eq-gbc11-1}\\
\delbar\phi=0&\ \iff\ \delbar f\wedge\omega+f\delbar\omega+\delbar\gamma=0,\label{eq-gbc11-2}\\
\delbar^*\del^*\phi=0&\ \iff\ \del\delbar*(f\omega+\gamma)=0.\label{eq-gbc11-3}
\end{align}
Expanding condition (\ref{eq-gbc11-3}), using condition (\ref{eq-gbc11-2}) and $\del\delbar\omega=0$, we get
\begin{align*}
0=\del\delbar*(f\omega+\gamma)&=\del\delbar({f}\omega-\gamma)\\
&=\del(\delbar f\wedge\omega+f\delbar\omega-\delbar\gamma)\\
&=2\del(\delbar f\wedge\omega+f\delbar\omega)\\
&=2\del\delbar{f}\wedge\omega-2\delbar{f}\wedge\del\omega+2\del{f}\wedge\delbar\omega.
\end{align*}

We claim that the differential operator $P:\cinf(M,\C)\to\cinf(M,\C)$ defined by
\[
P:f\mapsto -i*(\del\delbar{f}\wedge\omega-\delbar{f}\wedge\del\omega+\del{f}\wedge\delbar\omega)
\]
is strongly elliptic, since its principal part is given by 
\[
-i*(\del\delbar{f}\wedge\omega).
\]
Let us then verify that the differential operator $L:\cinf(M,\C)\to\cinf(M,\C)$ defined by
\[
L:f\mapsto -i*(\del\delbar f\wedge\omega)
\]
is strongly elliptic. 
Choose a local coframe $\{\zeta^1,\zeta^2\}$ of bi-degree $(1,0)$ centered in a point $m\in M$ and such that the almost Hermitian metric is written
\[
\omega=i(\zeta^{1\c1}+\zeta^{2\c2}).
\]
Let $\{V_1,V_2\}$ be the corresponding dual frame. We have
\begin{align*}
\del\delbar f&=\del(\c{V}_1(f)\c\zeta^1+\c{V}_2(f)\c\zeta^2)\\
&=V_1(\c{V}_1(f))\zeta^{1\c1}+V_2(\c{V}_1(f))\zeta^{2\c1}+\\
&+V_1(\c{V}_2(f))\zeta^{1\c2}+V_2(\c{V}_2(f))\zeta^{2\c2}+\\
&+\c{V}_1(f)\del\c\zeta^1+\c{V}_2(f)\del\c\zeta^2.
\end{align*}
Wedging $\del\delbar f$ together with $\omega$, we get
\begin{equation*}
\del\delbar f\wedge\omega=i\left(V_1(\c{V}_1(f))+V_2(\c{V}_2(f))+R(f)\right)\zeta^{1\c12\c2},
\end{equation*}
where $R$ is a differential operator which involves at most first order derivatives of $f$.
Since $\vol=\zeta^{12\c1\c2}$, it follows that
\[
L(f)=-\left(V_1(\c{V}_1(f))+V_2(\c{V}_2(f))+R(f)\right),
\]
which is strongly elliptic since $-V_1\c{V}_1-V_2\c{V}_2$ is strongly elliptic, proving the claim.
Note that $L(f)=i\langle\del\delbar f,\omega\rangle$ is equal, up to a factor $-2$, to the complex Laplacian by Gauduchon, \cite{Ga}.

Let us prove that the function $f\in\ker(P)$ is constant. Note that $P$ is a real differential operator, i.e., $P(\c{f})=\c{P(f)}$. Hence, $f\in\ker(P)$ iff $\real{f}\in\ker(P)$ and $\img{f}\in\ker(P)$. By considering $\real{f}$ and $\img{f}$ instead of $f$, for the moment we may assume that $f$ is real valued. Then, $f:M\to\R$ has a maximum and a minimum. Let $m_0\in M$ be a maximum point for $f$ and set $f(m_0)=N$. Let $U\ni m_0$ be a local chart and $r>0$ such that $\c{B(m_0,r)}\subset U$. The differential operator $P$ is strongly elliptic on $M$, therefore it is strongly elliptic on $B(m_0,r)$. Since $f\in\ker(P)$, by the maximum principle it follows that $f$ is constant on $B(m_0,r)$. Since $\{m\in M\,:\,f(m)=N\}$ is both open and closed, $f$ is constant on $M$.

Therefore, we have just proved
\[
\H^{1,1}_{BC}\subset\{f\omega+\gamma\in A^{1,1}\,|\,f\in\C,\ *\gamma=-\gamma,\ d(f\omega+\gamma)=0\}.
\]
Vice versa, if $\phi=f\omega+\gamma\in A^{1,1}$, with $f\in\C$, $*\gamma=-\gamma$, and $d(f\omega+\gamma)=0$, then a straightforward computation shows that \eqref{eq-gbc11-1}, \eqref{eq-gbc11-2} and \eqref{eq-gbc11-3} hold, providing the converse inclusion $\supset$. Therefore \eqref{eq-11-bc} is proved.

(II) Now, let us prove that either $h^{1,1}_{BC}=b^-$ or $h^{1,1}_{BC}=b^-+1$ holds. We have two possible cases: \newline
(a) there exists an element $f_0\omega+\gamma_0\in\H^{1,1}_{BC}$ such that 
$$f_0\in\C\setminus\{0\}, \quad *\gamma_0=-\gamma_0,\quad d(f_0\omega+\gamma_0)=0;
$$
(b) for any given element $f\omega+\gamma\in\H^{1,1}_{BC}$ we have $f=0$.\smallskip

In case (a), we claim that
\[
\H^{1,1}_{BC}=\{f(f_0\omega+\gamma_0)+\gamma\in A^{1,1}\,|\,f\in\C,\ *\gamma=-\gamma,\ d\gamma=0\},
\]
which yields $h^{1,1}_{BC}=b^-+1$. The inclusion $\supset$ is immediate. Indeed,
$$
d(f(f_0\omega+\gamma_0)+\gamma)=fd(f_0\omega+\gamma_0)+d\gamma=0,
$$
and $*(f\gamma_0+\gamma)=-f\gamma_0-\gamma$.\smallskip

To prove the converse inclusion $\subset$, let $f_1\omega+\gamma_1\in\H^{1,1}_{BC}$, i.e., $f_1\in\C$, $*\gamma_1=-\gamma_1$ and $d(f_1\omega+\gamma_1)=0$. We compute
\[
f_1\omega+\gamma_1=\frac{f_1}{f_0}(f_0\omega+\gamma_0)+\gamma_1-\frac{f_1}{f_0}\gamma_0=f(f_0\omega+\gamma_0)+\gamma,
\]
where we set $f=\frac{f_1}{f_0}$ and $\gamma=\gamma_1-\frac{f_1}{f_0}\gamma_0$. Note that 
$$f\in\C, \quad *\gamma=-\gamma,\quad d\gamma=-f_1d\omega+\frac{f_1}{f_0}f_0d\omega=0,
$$
proving the claim. 

In case (b), since for every element $f\omega+\gamma\in\H^{1,1}_{BC}$ we have $f=0$, it follows that
$\H^{1,1}_{BC}$ coincides with the space of harmonic and anti-self-dual $(1,1)$-forms $\H^-_{\C}$, yielding $h^{1,1}_{BC}=b^-$.\newline
The theorem is proved.
\end{proof}

We specialize Theorem \ref{thm-11} when the Hermitian metric is almost K\"ahler, yielding that $h^{1,1}_{BC}$ is independent of the choice of almost K\"ahler metrics $g$ compatible with $J$. The following corollary is the Bott-Chern analogue of \cite[Proposition 6.1]{HZ} by Holt and Zhang.

\begin{corollary}\label{cor-11}
Let $(M,J,g,\omega)$ be a compact almost K\"ahler manifold of real dimension $4$.
Then, $h^{1,1}_{BC}=b^-+1$ and $\H^{1,1}_{BC}=\H^{1,1}_{\delbar}$.
\end{corollary}
\begin{proof}
By the characterization \eqref{eq-11-bc} of Theorem \ref{thm-11} and $d\omega=0$ we get the characterization
\[
\H^{1,1}_{BC}=\{f\omega+\gamma\in A^{1,1}\,|\,f\in\C,\ *\gamma=-\gamma,\ d\gamma=0\},
\]
yielding $h^{1,1}_{BC}=b^-+1$. By \cite[Proposition 6.1]{HZ}, the space $\H^{1,1}_{\delbar}$ has the same characterization as $\H^{1,1}_{BC}$, implying $\H^{1,1}_{BC}=\H^{1,1}_{\delbar}$.
\end{proof}

By Remark \ref{conformal}, note that the same thesis of Corollary \ref{cor-11} holds if the almost Hermitian metric is conformal to an almost K\"ahler metric.

As another consequence of Theorem \ref{thm-11}, we derive the following

\begin{corollary}\label{cor-hodge}
Let $(M,J,g,\omega)$ be a compact almost Hermitian manifold of real dimension $4$.
If $\del\delbar\omega=0$ and $d\omega\ne0$, then every Hodge harmonic $(1,1)$-form is anti-self-dual.
\end{corollary}
\begin{proof}
Assume $\del\delbar\omega=0$ and $d\omega\ne0$, and take $\phi\in A^{1,1}$ such that $\Delta_d\phi=0$. 
As in Theorem \ref{thm-11}, by equation (\ref{eq-dec11}), we have $\phi=f\omega+\gamma$, where $f$ is a smooth function with complex values on $M$ and $\gamma$ is anti-self-dual, i.e., $*\gamma=-\gamma$. Recall $\Delta_d\phi=0$ if and only if $d\phi=d*\phi=0$. This implies $\phi\in\H^{1,1}_{BC}$, since $d=\del+\delbar$ on $(1,1)$-forms, and $d*\phi=0$ yields $\del\delbar*\phi=0$.
By the proof of Theorem \ref{thm-11}, $f$ is a complex constant. Computing $d\phi=d*\phi=0$, we find
\begin{gather*}
0=d\phi=fd\omega+d\gamma,\\
0=d*\phi=fd\omega-d\gamma,
\end{gather*}
hence
\[
fd\omega=0,\ \ \ d\gamma=0.
\]
Since $d\omega\ne0$, we get $f=0$.
\end{proof}

\section{Bott-Chern harmonic forms on the Kodaira-Thurston manifold}\label{bc-harmonic-kodaira}

In this section we are going to compare Bott-Chern and Dolbeault harmonic forms on a family of almost K\"ahler structures on the Kodaira-Thurston manifold, following Holt and Zhang, \cite{HZ}.

The Kodaira-Thurston manifold, here denoted by $M$, is defined to be the direct product $S^1\times(H_3(\Z)\backslash H_3(\R))$, where $H_3(\R)$ denotes the Heisenberg group
\begin{equation*}
H_3(\R)=\left\{
\begin{pmatrix}1&x&z\\0&1&y\\0&0&1\end{pmatrix}
\in\GL(3,\R)\right\},
\end{equation*}
and $H_3(\Z)$ is the subgroup $H_3(\R)\cap\GL(3,\Z)$, acting on $H_3(\R)$ by left multiplication. The manifold $M$ is compact and connected. If $t$ is the coordinate on the circle $S^1$, and $x,y,z$ are the coordinates on $H_3(\Z)\backslash H_3(\R)$ as in the definition of $H_3(\R)$, we see that the manifold $M$ can be identified with $\R^4$, endowed with the group structure of $\R\times H_3(\R)$, quotiented by the equivalence relation
\begin{equation*}
\begin{pmatrix}t\\x\\y\\z\end{pmatrix}\sim
\begin{pmatrix}t+t_0\\x+x_0\\y+y_0\\z+z_0+x_0y\end{pmatrix},
\end{equation*}
for every $t_0,x_0,y_0,z_0\in\Z$. The vector fields
\begin{equation*}
e_1=\de{}{t},\ e_2=\de{}{x},\ e_3=\de{}{y}+x\de{}{z},\ e_4=\de{}{z}
\end{equation*}
are left invariant and form a basis of $T_pM$ at each point $p\in M$. The dual left invariant coframe is given by
\begin{equation*}
e^1=dt,\ e^2=dx, e^3=dy,\ e^4=dz-xdy.
\end{equation*}

Consider the almost complex structure $J_{b}$, for $b\in\R\setminus\{0\}$, given by
\begin{equation*}
V_1=\frac12\left(\de{}{t}-i\de{}{x}\right),\ \ \ V_2=\frac12\left(\left(\de{}{y}+x\de{}{z}\right)+\frac{i}{b}\de{}{z}\right)
\end{equation*}
spanning $T_p^{1,0}M$ at every point $p\in M$, along with their dual  $(1,0)$-forms
\begin{equation*}
\phi^1=dt+idx,\ \ \ \phi^2=dy-ib(dz-xdy).
\end{equation*}
Their structure equations are
\begin{equation*}
d\phi^1=0,\ \ \ d\phi^2=\frac{b}4(\phi^{12}+\phi^{1\c2}+\phi^{2\c1}-\phi^{\c1\c2}).
\end{equation*}

Endow every $(M,J_{b})$ with the family of almost K\"ahler metrics given by the compatible symplectic forms
\begin{equation*}
\omega_{b}=i(\phi^1\wedge\c\phi^1+\phi^2\wedge\c\phi^2)=2dt\wedge dx+2 b dz\wedge dy.
\end{equation*}
 Define the volume form $\vol$ such that
\begin{equation*}
2\vol=\omega_{b}^2=2\phi^1\wedge\phi^2\wedge\c\phi^1\wedge\c\phi^2=8 bdt\wedge dx\wedge dz\wedge dy.
\end{equation*}
Holt and Zhang, in \cite{HZ}, computed the spaces $\H^{p,q}_{\delbar}$ for every $p,q$. We will verify when $\H^{p,q}_{BC}=\H^{p,q}_{\delbar}$ and show that this equality is not true for every $p,q$ on the Kodaira-Thurston manifold $M$. Note that $\H^{p,q}_{\delbar}=\c{*\H^{2-p,2-q}_{\delbar}}$ by Serre duality. For Bott-Chern harmonic forms, we have
\begin{equation}\label{bc-cong}
\H^{p,q}_{BC}=\c{\H^{q,p}_{BC}}
\end{equation}
if it holds $\del\delbar+\delbar\del=0$ when restricted on $(n-q,n-p)$ forms.
For an almost complex $4$-manifold, \eqref{bc-cong} is true when
\begin{equation*}
(p,q)\in\{(0,0),(2,2),(2,0),(0,2),(1,0),(0,1),(1,1)\}.
\end{equation*}

Let us compute the spaces $\H^{p,q}_{BC}$, $\H^{p,q}_{\delbar}$ and compare them.
\subsection*{Bidegrees $(0,0),(2,2)$}
For $p=q=0$, it is immediate to see that both spaces are equal to constant functions on $M$. Similarly, for $p=q=2$, both spaces are spanned by $\phi^{12\c1\c2}$.

\subsection*{Bidegrees $(2,0),(0,2)$}
For $(p,q)\in\{(2,0),(0,2)\}$, note that $\H^{2,0}_{\delbar}=\c{*\H^{0,2}_{\delbar}}$ and $\H^{2,0}_{BC}=\c{\H^{0,2}_{BC}}$, therefore it is sufficient to prove $\H^{2,0}_{BC}=\H^{2,0}_{\delbar}$. This follows immediately since both spaces turn out to be equal to the set of $\delbar$-closed $(2,0)$-forms, i.e.,
\begin{equation*}
\H^{2,0}_{BC}=\H^{2,0}_{\delbar}=\{\alpha\in A^{2,0}(M)\,|\,\delbar\alpha=0\}.
\end{equation*}

In \cite[Section 6]{CZ} it is proved that
\begin{equation*}
\H^{2,0}_{\delbar}=
\begin{cases}
\C<\phi^{12}>&\text{ if }0\ne b\in 4\pi\Z,\\
0&\text{ if }b\notin 4\pi\Z.
\end{cases}
\end{equation*}

\subsection*{Bidegree $(1,1)$}
For $(p,q)=(1,1)$, in \cite{HZ} Holt and Zhang proved that on every almost complex $4$-manifold with an almost K\"ahler metric, every $(1,1)$-Dolbeault harmonic form is the sum of a complex multiple of the almost K\"ahler symplectic form and of an anti-self-dual harmonic form. Corollary \ref{cor-11} affirms that the same holds for every $(1,1)$-Bott-Chern harmonic forms.
Since $b^-=2$, it follows that $h^{1,1}_{\delbar}=h^{1,1}_{BC}=3$ and it suffices to find three harmonic $(1,1)$-forms to prove
\begin{equation*}
\H^{1,1}_{BC}=\H^{1,1}_{\delbar}=\C<\phi^{1\c1},\phi^{2\c2},\phi^{1\c2}-\phi^{2\c1}>.
\end{equation*}

\subsection*{Bidegrees $(1,0),(0,1)$}
For $(p,q)\in\{(1,0),(0,1)\}$, since $\H^{1,0}_{BC}=\c{\H^{0,1}_{BC}}$, let us compute $\H^{0,1}_{BC}$.
Let $s=f\c\phi^1+g\c\phi^2\in \H^{0,1}_{BC}$, i.e., $\del s=0$ and $\delbar s=0$. We get
\begin{align}
\del s=0\ &\iff\ V_1(f)\phi^{1\c1}+(V_1(g)-\frac{b}4g)\phi^{1\c2}+(V_2(f)-\frac{b}4g)\phi^{2\c1}+V_2(g)\phi^{2\c2},\label{01-1}\\
\delbar s=0\ &\iff\ (-\c{V_2}(f)+\c{V_1}(g)+\frac{b}4g)\phi^{\c1\c2}\label{01-2}=0.
\end{align}

By equation (\ref{01-1}), note that $V_2(g)=0$, implying $\c{V_2}V_2(g)=0$. The operator $-\c{V_2}V_2$ is a real operator, and it is strongly elliptic when computed on functions depending only on the coordinates $y,z$. Consider the projection $\pi:M\to T^2=\Z^2\backslash\R^2$ given by $\pi([t,x,y,z])=([t,x])$. The fiber of $\pi$ is a torus with coordinates $y,z$. As the fiber is compact, by the maximum principle applied to $\real(g)$ and $\img(g)$, we get that $g$ is constant on each fiber. Therefore, the function $g$ on $M$ depends only on the coordinates $t,x$. 

Now, note that $4[\c{V_1},V_2]=b(V_2-\c{V_2})$. Therefore, applying $V_2$ to 
\[
-\c{V_2}(f)+\c{V_1}(g)+\frac{b}4g=0,
\]
which derives from \eqref{01-2}, and taking into account that $g=g(t,x)$, we obtain $V_2\c{V_2}(f)=0$.

Again by (\ref{01-1}), note that $V_1(f)=0$, implying $\c{V_1}V_1(f)=0$. Since $[V_2,\c{V_2}]=0$, we derive that $f$ belongs to the kernel of $-\c{V_1}V_1-\c{V_2}V_2$, which is strongly elliptic. It follows that $f$ is a complex constant.

Equation (\ref{01-1}) also yields
\begin{equation*}
V_2(f)-\frac{b}4g=0.
\end{equation*}
Since $f$ is a complex constant, it follows that $g=0$.

Therefore 
\begin{equation*}
\H^{0,1}_{BC}=\C<\c\phi^1>,\ \ \ \H^{1,0}_{BC}=\C<\phi^1>.
\end{equation*}
Analogously, since $\H^{1,0}_{\delbar}=\{\alpha\in A^{1,0}(M)\,|\,\delbar\alpha=0\}$, it is easy to see that
\begin{equation*}
\H^{1,0}_{\delbar}=\C<\phi^1>\,=\H^{1,0}_{BC},
\end{equation*}
see \cite[Section 6]{CZ} for the proof. It is also easy to see 
\begin{equation*}
\C<\c\phi^1>\,\subset\H^{0,1}_{\delbar}.
\end{equation*}
 However, in \cite{HZ}, Holt and Zhang proved that $\sigma=Ce^{2\pi i l x}\c\phi^{2}\in\H^{0,1}_{\delbar}$, for $l\in\Z$, $b=4\pi l$ and for any $C\in\C$. Since $\sigma\notin\H^{0,1}_{BC}$, we just proved, for $b=4\pi l$,
\begin{equation*}
\H^{0,1}_{BC}\underset{\ne}{\subset}\H^{0,1}_{\delbar}.
\end{equation*}

\subsection*{Bidegree $(2,1)$}
For $(p,q)=(2,1)$, note that $\H^{2,1}_{\delbar}=\c{*\H^{0,1}_{\delbar}}$. Therefore, for $b=4\pi l$, we know $\c{*\sigma}\in\H^{2,1}_{\delbar}$. We will show that $\c{*\sigma}\notin\H^{2,1}_{BC}$, implying
\begin{equation*}
\H^{2,1}_{BC}\ne\H^{2,1}_{\delbar}.
\end{equation*}

Let us begin by describing the space $\H^{2,1}_{BC}$. Let $s=f\phi^{12\c1}+g\phi^{12\c2}$, then $s\in\H^{2,1}_{BC}$ if and only if  $\del\delbar*s=0$ and $\delbar s=0$, i.e., iff
\begin{equation}\label{bc-21}
\begin{cases}
V_1\c{V_1}(f)+V_2\c{V_1}(g)-\frac{b}4V_1(f)+\frac{b}4\c{V_1}(f)-\frac{b}4\c{V_2}(g)-\frac{b^2}8f=0,\\
V_1\c{V_2}(f)+V_2\c{V_2}(g)+\frac{b}4V_2(f)=0,\\
\c{V_1}(g)-\c{V_2}(f)=0.
\end{cases}
\end{equation}

It is an easy verification that the $(2,1)$-form $\c{*\sigma}$ does not satisfy the first equation of the system (\ref{bc-21}).

\subsection*{Bidegree $(1,2)$}
For $(p,q)=(1,2)$, we know $\H^{1,2}_{\delbar}=\c{*\H^{1,0}_{\delbar}}=\C<\phi^{2\c1\c2}>$. We will show that for some value of $b\ne0$ it holds
\begin{equation}\label{claim-12}
\H^{1,2}_{\delbar}\underset{\ne}{\subset}\H^{1,2}_{BC}.
\end{equation}

Let $s=f\phi^{1\c1\c2}+g\phi^{2\c1\c2}\in\H^{1,2}_{BC}$, i.e.,  $\del\delbar*s=0$ and $\del s=0$, i.e.,
\begin{equation}\label{bc-12}
\begin{cases}
V_1\c{V_1}(f)+V_1\c{V_2}(g)+\frac{b}4V_1(f)-\frac{b}4\c{V_1}(f)-\frac{b}4\c{V_2}(g)-\frac{b^2}{16}f=0,\\
 V_2\c{V_1}(f)+V_2\c{V_2}(g)+\frac{b}4V_2(f)=0,\\
{V_1}(g)-{V_2}(f)=0.
\end{cases}
\end{equation}
To prove (\ref{claim-12}), it will be sufficient to study solutions $f,g$ of system (\ref{bc-12}) which only depend on coordinates $t,x,y$. 
We decompose the functions $f,g$ in Fourier series as
\begin{equation*}
f=\sum_{k,l,m\in\Z}f_{k,l,m}e^{2\pi i(kt+lx+my)},\ \ \ g=\sum_{k,l,m\in\Z}g_{k,l,m}e^{2\pi i(kt+lx+my)}.
\end{equation*}

System (\ref{bc-12}) rewrites into
\begin{numcases}{}
(16\pi^2(k^2+l^2)-8b\pi l+b^2)f_{k,l,m}+4\pi m(4\pi k-4\pi il+ ib)g_{k,l,m}=0,     \label{f1}\\
\pi  m(4\pi k+4\pi il-ib)f_{k,l,m}+4\pi^2m^2g_{k,l,m}=0,      \label{f2}\\
m f_{k,l,m}=(k-il)g_{k,l,m},      \label{f3}
\end{numcases}
for all $k,l,m\in\Z$. From (\ref{f2}) and (\ref{f3}), we obtain $k=0$ and
\begin{equation}\label{sol-int}
4\pi m^2+4\pi l^2-bl=0,
\end{equation}
and, once we impose these conditions, (\ref{f1}) reduces to (\ref{f3}). 

Summing up, for any $l,m\in\Z$ and $b\ne0$ such that $4\pi m^2+4\pi l^2-bl=0$, there exists $s\in\H^{1,2}_{BC}$ given by
\begin{equation}\label{solutions}
\begin{cases}
s=Ce^{2\pi i(lx+my)}\phi^{1\c1\c2}-C\frac{im}{l}e^{2\pi i(lx+my)}\phi^{2\c1\c2}
&\text{ if }l\ne0,\\
s=C\phi^{2\c1\c2}&\text{ if }l=0,
\end{cases}
\end{equation}
for any $C\in\C$. For $l=0$, we get the inclusion of (\ref{claim-12}); to show that the inclusion is not an equality, take, e.g.,  $b=8\pi$, $l=1$, $m=1$.

\begin{remark}
Counting the solutions (\ref{solutions}), i.e., finding a lower bound on the complex dimension of $\H^{1,2}_{BC}$, is equivalent to asking how many couples $(l,m)\in\Z^2$ satisfy \eqref{sol-int},
which is equivalent to counting how many couples $(l,m)\in\Z^2$ satisfy
\begin{equation*}
m^2+(l-d)^2=d^2,
\end{equation*}
where we set $d=b/8\pi$. Counting the number of solutions can be thought of as asking how many lattice
points in $\Z\times\Z$ lie on a circle with centre $(d, 0)$ and radius $d$. This last number theoretic problem has already been addressed and solved by Holt and Zhang in \cite[Section 4]{HZ}, where they show that by changing the choice of $b$ (or equivalently $d$) one can make the number of solutions
become arbitrarily large.

Therefore, in view of the argument as above, we infer that by changing our choice of $b$, $h^{1,2}_{BC}$ may become arbitrarily large.
This conclusion has been already obtained by Holt in \cite[Example 4.4]{Ho} (the case $\rho=1$ in the notation of Holt), where the space of Bott-Chern harmonic $(1,2)$-forms is fully characterized.
\end{remark}

Summarizing the results just obtained, we state the following proposition.

\begin{proposition}\label{cor-kod}
Let $M$ be the Kodaira-Thurston manifold $S^1\times(H_3(\Z)\backslash H_3(\R))$ with local coordinates $t,x,y,z$.
Consider the almost complex structure $J_{b}$, for $b\in\R\setminus\{0\}$, given by
\begin{equation*}
\phi_1=dt+idx,\ \ \ \phi_2=dy-ib(dz-xdy).
\end{equation*}
spanning $(T_p^{1,0}M)^*$ at every point $p\in M$.
Endow $(M,J_{b})$ with the almost K\"ahler metric given by the compatible symplectic form
\begin{equation*}
\omega_{b}=i(\phi^1\wedge\c\phi^1+\phi^2\wedge\c\phi^2)=2dt\wedge dx+2b dz\wedge dy,
\end{equation*}
and the volume form $\vol$ such that $2\vol=\omega_{b}^2$.
Then, for all $b\in\R\setminus\{0\}$, and $(p,q)\in\{(0,0),(1,0),(1,1),(2,0),(0,2),(2,2)\}$
\begin{equation*}
\H^{p,q}_{BC}=\H^{p,q}_{\delbar},
\end{equation*}
while for $(p,q)\in\{(0,1),(2,1),(1,2)\}$, there exists $b\in\R\setminus\{0\}$ such that
\begin{equation*}
\H^{p,q}_{BC}\ne\H^{p,q}_{\delbar}.
\end{equation*}
\end{proposition}

\begin{corollary}\label{cor-kod-2}
There exists an almost K\"ahler $4$-manifold $(M,J,g,\omega)$ such that for some bidegree $(p,q)$ it holds that
\begin{equation*}
\H^{p,q}_{BC}\ne\H^{p,q}_{\delbar}.
\end{equation*}
\end{corollary}

\section{Locally conformally almost K\"ahler metrics}\label{sec-lck}
Let $(M,J,g,\omega)$ be an almost Hermitian manifold. Following \cite{TT}, we say $\omega$ is \emph{strictly locally conformally almost K\"ahler} if
\[
d\omega=\theta\wedge\omega,
\]
and $\theta$ is $d$-closed but non $d$-exact. Conversely, we say $\omega$ is \emph{globally conformally almost K\"ahler}, if
\[
d\omega=\theta\wedge\omega,
\]
and $\theta$ is $d$-exact. As mentioned in the introduction, Tardini and the second author proved that $h^{1,1}_{\delbar}=b^-$ on every compact almost complex 4-manifold with a strictly locally conformally almost K\"ahler metric. Here, we prove that Bott-Chern harmonic $(1,1)$-forms have a different behaviour than Dolbeault harmonic $(1,1)$-forms. Namely, we describe an almost complex structure on a hyperelliptic surface, endowed with a strictly locally conformally almost K\"ahler metric, such that $h^{1,1}_{BC}=b^-+1$.

Note that in the integrable case, i.e., on compact complex surfaces, it holds $h^{1,1}_{BC}=b^-+1$ on K\"ahler surfaces, on complex surfaces diffeomorphic to solvmanifolds, and on complex surfaces of class VII (see \cite{ADT} and \cite[Chapter IV, Theorem 2.7]{BPV}).

Following Hasegawa, \cite{Ha}, let $G$ be the group $\C^2$ together with the multiplication
\[
(w^1,w^2)\cdot(z^1,z^2)=(w^1+e^{i\pi\frac{w^2+\c{w^2}}2}z^1,w^2+z^2),
\]
and let $\Gamma$ be the subgroup of $G$ given by $(\Z+i\Z)^2$. This corresponds to the hyperelliptic surface with $\eta=\pi$ and $p=q=s=t=0$ in the notation of Hasegawa. Let $M$ be the solvmanifold $\Gamma\backslash G$, and denote by $x^1,y^1,x^2,y^2$ the local coordinates of $M$ induced from $\C^2$, i.e., $z^1=x^1+iy^1$, $z^2=x^2+iy^2$. The vector fields
\begin{gather*}
e_1=\cos(\pi x^2)\de{}{x^1}+\sin(\pi x^2)\de{}{y^1},\\ e_2=-\sin(\pi x^2)\de{}{x^1}+\cos(\pi x^2)\de{}{y^1},\\ e_3=\de{}{x^2},\ e_4=\de{}{y^2}
\end{gather*}
are left invariant with respect to the action of
the subgroup $\Gamma$ and form a basis of $TM$ at each point. The dual left invariant coframe is given by
\begin{gather*}
e^1=\cos(\pi x^2)dx^1+\sin(\pi x^2)dy^1,\\ e^2=-\sin(\pi x^2)dx^1+\cos(\pi x^2)dy^1,\\ e^3=dx^2,\ e^4=dy^2,
\end{gather*}
with structure equations
\[
de^1=-\pi e^{23},\ de^2=\pi e^{13},\ de^3=0,\ de^4=0.
\]
The De Rham cohomology of $M$ is, see e.g. \cite{ADT},
\begin{equation}\label{de-rham}
H_{dR}^1=\R<e^3,e^4>,\ \ \ H_{dR}^2=\R<e^{12},e^{34}>.
\end{equation}

Consider the almost complex structure $J$ given by
\begin{gather*}
V_1=\frac12(e_1-ie_3)=\frac12(\cos(\pi x^2)\de{}{x^1}+\sin(\pi x^2)\de{}{y^1}-i\de{}{x^2}),\\
V_2=\frac12(e_2-ie_4)=\frac12(-\sin(\pi x^2)\de{}{x^1}+\cos(\pi x^2)\de{}{y^1}-i\de{}{y^2}),
\end{gather*}
spanning $T_p^{1,0}M$ at every point $p\in M$, along with their dual $(1,0)$-forms
\begin{gather*}
\phi^1=e^1+ie^3=\cos(\pi x^2)dx^1+\sin(\pi x^2)dy^1+idx^2,\\
\phi^2=e^2+ie^4=-\sin(\pi x^2)dx^1+\cos(\pi x^2)dy^1+idy^2.
\end{gather*}
Their structure equations are
\begin{equation*}
d\phi^1=i\frac\pi4(-\phi^{12}-\phi^{1\c2}-\phi^{2\c1}+\phi^{\c1\c2}),\ \ \ d\phi^2=i\frac\pi2\phi^{1\c1}.
\end{equation*}

Endow $(M,J)$ with the almost Hermitian metric given by the compatible symplectic form
\begin{equation*}
\omega=e^{13}+e^{24}=\frac{i}2(\phi^1\wedge\c\phi^1+\phi^2\wedge\c\phi^2),
\end{equation*}
and define the volume form $\vol$ such that
\begin{equation*}
\vol=\frac{\omega^2}2=\frac14\phi^1\wedge\phi^2\wedge\c\phi^1\wedge\c\phi^2.
\end{equation*}
Note that $b^-=1$.
Also note that $\omega$ is strictly locally conformally almost K\"ahler, since
\[
d\omega=\pi e^{134}=\theta\wedge\omega,
\]
with $\theta=\pi e^4$, which is closed but not exact by \eqref{de-rham}.

Let us now compute $h^{1,1}_{BC}$. Let $\psi=f\omega+\gamma\in \H^{1,1}_{BC}$, with $f\in\C$ and $*\gamma=-\gamma$. The $(1,1)$-form $\gamma$ can be written as
\[
\gamma=A\phi^{1\c1}+B\phi^{1\c2}+C\phi^{2\c1}-A\phi^{2\c2},
\]
with $A,B,C\in\cinf(M,\C)$. We compute $d\psi=0$ and find
\begin{equation}\label{hyp-11}
\begin{cases}
4V_1(C)-4V_2(A)-2\pi iA-\pi f=0,\\
4V_1(A)+4V_2(B)+\pi iB+\pi iC=0,\\
4\c{V_1}(B)-4\c{V_2}(A)+2\pi iA+\pi f=0,\\
4\c{V_1}(A)+4\c{V_2}(C)-\pi i C-\pi iB=0.
\end{cases}
\end{equation}
Note that for $2A=if\ne0$ and $B=C=0$, we get $f\omega+\gamma\in \H^{1,1}_{BC}$ with $f\ne0$. By the proof of Theorem \ref{thm-11}, it immediately yields that $h^{1,1}_{BC}=b^-+1$. However, let us also reprove $h^{1,1}_{BC}=b^-+1$ explicitly, without the help of Theorem \ref{thm-11}.

Every function on $M$ is, in particular, $2(\Z+i\Z)$-periodic in both complex variables, therefore we may decompose the functions $A,B,C$ in Fourier series as
\begin{gather*}
A=\sum_{k,l,m,n\in\Z}A_{k,l,m,n}e^{i\pi (kx^1+ly^1+mx^2+ny^2)},\\ B=\sum_{k,l,m,n\in\Z}B_{k,l,m,n}e^{i\pi (kx^1+ly^1+mx^2+ny^2)},\\ C=\sum_{k,l,m,n\in\Z}C_{k,l,m,n}e^{i\pi (kx^1+ly^1+mx^2+ny^2)}.
\end{gather*}
For every $(0,0,0,0)\ne(k,l,m,n)\in\Z^4$ and $x^2\in\R$, system (\ref{hyp-11}) rewrites into
\begin{equation}\label{hyp-11-fou}
\begin{cases}
( i\cos(\pi x^2)l+ i\sin(\pi x^2)k+ m)C_{k,l,m,n}+\\
+( i\sin(\pi x^2) l- i \cos(\pi x^2)k- n- i)A_{k,l,m,n}=0,\\
(2 i\cos(\pi x^2)l+2 i\sin(\pi x^2)k+2 m)A_{k,l,m,n}+\\
+(-2 i\sin(\pi x^2) l+2 i \cos(\pi x^2)k+2 n+ i)B_{k,l,m,n}+ i C_{k,l,m,n}=0,\\
( i\cos(\pi x^2)l+ i\sin(\pi x^2)k- m)B_{k,l,m,n}+\\
+( i\sin(\pi x^2) l- i \cos(\pi x^2)k+ n+ i)A_{k,l,m,n}=0,\\
(2 i\cos(\pi x^2)l+2 i\sin(\pi x^2)k-2 m)A_{k,l,m,n}+\\
+(-2 i\sin(\pi x^2) l+2 i \cos(\pi x^2)k-2 n- i)C_{k,l,m,n}- i B_{k,l,m,n}=0.\\
\end{cases}
\end{equation}
Differentiate (for sign convenience)  system \eqref{hyp-11-fou} two times with respect to $x^2$ to find
\begin{equation}\label{hyp-11-fou-der}
\begin{cases}
(-\cos(\pi x^2)l-\sin(\pi x^2)k)C_{k,l,m,n}+(-\sin(\pi x^2) l+\cos(\pi x^2)k)A_{k,l,m,n}=0,\\
(-\cos(\pi x^2)l-\sin(\pi x^2)k)A_{k,l,m,n}+(\sin(\pi x^2) l-\cos(\pi x^2)k)B_{k,l,m,n}=0,\\
(-\cos(\pi x^2)l-\sin(\pi x^2)k)B_{k,l,m,n}+(-\sin(\pi x^2) l+\cos(\pi x^2)k)A_{k,l,m,n}=0,\\
(-\cos(\pi x^2)l-\sin(\pi x^2)k)A_{k,l,m,n}+(\sin(\pi x^2) l-\cos(\pi x^2)k)C_{k,l,m,n}=0.\\
\end{cases}
\end{equation}
If $x^2=0$, it is easy to see that $(k,l)\ne(0,0)$ implies $A_{k,l,m,n}=B_{k,l,m,n}=C_{k,l,m,n}=0$.
Therefore, the functions $A,B,C$ depend only on variables $x^2,y^2$, and we can assume $k=l=0$.
For every $(0,0)\ne(m,n)\in\Z^2$, system (\ref{hyp-11-fou}) rewrites into
\begin{equation}\label{hyp-11-fou-2}
\begin{cases}
 mC_{0,0,m,n}-( n+ i)A_{0,0,m,n}=0,\\
2 mA_{0,0,m,n}+(2n+ i)B_{0,0,m,n}+ i C_{0,0,m,n}=0,\\
 mB_{0,0,m,n}-( n+ i)A_{0,0,m,n}=0,\\
2 mA_{0,0,m,n}+(2 n+ i)C_{0,0,m,n}+ i B_{0,0,m,n}=0.\\
\end{cases}
\end{equation}
From system \eqref{hyp-11-fou-2}, subtracting the third equation from the first, we get 
\[
m(C_{0,0,m,n}-B_{0,0,m,n})=0.
\]
If $m=0$, then $A_{0,0,0,n}=B_{0,0,0,n}=C_{0,0,0,n}=0$ for every $0\ne n\in\Z$. Conversely, if $C_{0,0,m,n}=B_{0,0,m,n}$ for all $m,n\in\Z$ with $m\ne0$, combining the first two equations of \eqref{hyp-11-fou-2} we get
\[
(m^2+n^2-1+2ni)B_{0,0,m,n}=0,
\]
implying either $n=0$ and $m^2=1$ or $B_{0,0,m,n}=0$ for $n\ne0$. If $B_{0,0,m,n}=0$ for all $m,n\in\Z$ with $m\ne0\ne n$, it also follows that $A_{0,0,m,n}=C_{0,0,m,n}=0$. On the other hand, if $n=0$ and $m=\pm1$, then every choice $B_{0,0,m,0}=C_{0,0,m,0}=imA_{0,0,m,0}\in\C$ provides a solution of system (\ref{hyp-11-fou}). Therefore
\begin{equation}\label{eq-sol-abc}
B=C=\pm iA=\pm iKe^{\pm i\pi x^2}
\end{equation}
are solutions of system \eqref{hyp-11} for every complex constant $K\in\C$. However, note that
\[
A(z^1,z^2)=Ke^{\pm i\pi x^2}\ne Ke^{\pm i\pi (x^2+1)}=A(-z^1,z^2+1)=A((0,1)\cdot(z^1,z^2))
\]
for every $z^1=x^1+iy^1,z^2=x^2+iy^2\in\C$ and $0\ne K\in\C$, thus the functions $A,B,C$ in \eqref{eq-sol-abc} are not well defined on $M$.

Other solutions of system \eqref{hyp-11} are found when $k=l=m=n=0$ and $A,B,C\in\C$ are complex constants. More precisely, we get $2A=if$ and $B=-C$.

Therefore, we re-obtain $h^{1,1}_{BC}=2=b^-+1$ and
\[
\H^{1,1}_{BC}=\C<\phi^{1\c1},\phi^{1\c2}-\phi^{2\c1}>.
\]

\begin{remark}
It is worth asking if on compact almost complex 4-manifolds $h^{1,1}_{BC}$ may be always equal to $b^-+1$ or there are explicit examples where $h^{1,1}_{BC}=b^-$.
Very recently Holt, in \cite[Theorem 4.2]{Ho}, proved that $h^{1,1}_{BC}$ is always equal to $b^-+1$.
\end{remark}

\section{Bott-Chern cohomology of almost complex manifolds}\label{cohomology}
In \cite{CW}, Cirici and Wilson introduced a generalization of Dolbeault cohomology on almost complex manifolds. 
Let $(M,J)$ be an almost complex manifold and
\[
H^{p,q}_ {\c\mu}=\frac{\ker\c\mu\cap A^{p,q}}{\c\mu A^{p+1,q-2}}
\]
be the $\c\mu$-cohomology, which is well defined since $\c\mu^2=0$.
Note that $\delbar$ induces a morphism of vector spaces
\[
\delbar:H^{p,q}_ {\c\mu}\to H^{p,q+1}_ {\c\mu},
\]
since $\c\mu\delbar+\delbar\c\mu=0$. Furthermore, $\delbar^2+\c\mu\del+\del\c\mu=0$ implies $\delbar^2=0$ on $H^{p,q}_ {\c\mu}$. Then, the Dolbeault cohomology of $M$ is defined by
\[
H^{p,q}_{Dol}=\frac{\ker\delbar\cap H^{p,q}_ {\c\mu}}{\delbar H^{p,q-1}_ {\c\mu}}.
\]
Analogously, define the $\mu$-cohomology
\[
H^{p,q}_ {\mu}=\frac{\ker\mu\cap A^{p,q}}{\mu A^{p-2,q+1}},
\]
and the conjugated Dolbeault cohomology
\[
H^{p,q}_{\c{Dol}}=\frac{\ker\del\cap H^{p,q}_ {\mu}}{\del H^{p-1,q}_ {\mu}}.
\]
The Dolbeault cohomology of almost complex manifolds generalizes the classical Dolbeault cohomology of complex manifolds, and satisfies some desirable properties. In particular, the authors modify the classical Hodge filtration for complex manifolds by taking into account the presence of $\c\mu$ and show that the Dolbeault cohomology of every almost complex manifold arises in the first stage of the spectral sequence associated to this new Hodge filtration, which converges to the complex de Rham cohomology of the manifold.
However, in \cite{CPS}, Coelho, Placini and Stelzig show that the Dolbeault cohomology of almost complex manifolds is often infinite dimensional.

Still in \cite{CPS}, the authors also give the following definition for Bott-Chern and Aeppli cohomologies of almost complex manifolds. 
Given any almost complex manifold $(M,J)$, consider the spaces of forms
\[
A^{*,*}_s:=\ker \c\mu\cap\ker\delbar^2\cap\ker\del^2\cap\ker\mu
\]
and
\[
A^{*,*}_r:=A^{*,*}/(\im\c\mu+\im\delbar^2+\im\del^2+\im\mu),
\]
and note that both $(A_s^{*,*},\del,\delbar)$ and $(A_r^{*,*},\del,\delbar)$ are double complexes. Therefore, define the Bott-Chern and Aeppli cohomologies of an almost complex manifold as the usual Bott-Chern and Aeppli cohomologies of respectively the double complexes $(A_s^{*,*},\del,\delbar)$ and $(A_r^{*,*},\del,\delbar)$. More precisely,
\[
H^{p,q}_{BC}:=\frac{\ker d\cap A^{p,q}_s}{\del\delbar A^{p-1,q-1}_s}
\]
and
\[
H^{p,q}_{A}:=\frac{\ker \del\delbar\cap A^{p,q}_r}{\del A^{p-1,q}_r+\delbar A^{p,q-1}_r}.
\]
It turns out the following commutative diagram holds  as in the integrable case
\[
\begin{tikzcd}
\  & H^{*,*}_{BC} \arrow[ld] \arrow[d] \arrow[rd] &\  \\
H^{*,*}_{Dol}  \arrow[rd]  & H^{*}_{dR} \arrow[d]   &  \arrow[ld] H^{*,*}_{\c{Dol}}\\
\ & H^{*,*}_{A}&\ 
\end{tikzcd}
\]
where arrows are morphisms of vector spaces. Moreover, $H^{*,*}_A$ is a bigraded module over $H^{*,*}_{BC}$, and conjugation induces isomorphisms $H^{p,q}_{BC}\cong H^{q,p}_{BC}$, $H^{p,q}_{A}\cong H^{q,p}_{A}$.  Note that, like Dolbeault cohomology, also Bott-Chern and Aeppli cohomologies may be infinite dimensional on compact almost complex manifolds.

Let $(M,J,g,\omega)$ be a compact almost Hermitian manifold. Since the Bott-Chern and Aeppli Laplacians are elliptic, the Hodge theory developed by Schweitzer in \cite{S} applies, yielding the $L^2$-orthogonal decompositions
\[
A^{p,q}=\H^{p,q}_{BC}\oplus\del\delbar A^{p-1,q-1}\oplus(\del^* A^{p+1,q}+\delbar^* A^{p,q+1})
\]
and
\[
A^{p,q}=\H^{p,q}_{A}\oplus\delbar^*\del^* A^{p+1,q+1}\oplus(\del A^{p-1,q}+\delbar A^{p,q-1}).
\]

In general, the spaces $\H^{p,q}_{BC}$ and $\H^{p,q}_{A}$ seem to be unrelated to the Bott-Chern and Aeppli cohomology spaces just introduced. 

However, if we take $n=\dim_\R M=4$ and $p=q=1$, note that $A^{1,1}_s=A^{1,1}=A^{1,1}_r$, and the previous Bott-Chern decomposition yields, intersecting with $\ker d$,
\[
\ker d\cap A^{1,1}=\H^{1,1}_{BC}\oplus\ker d\cap\del\delbar A^{0,0}.
\]
Therefore, there is a well defined injection
\[
\H^{1,1}_{BC}\overset{j}{\longrightarrow} H^{1,1}_{BC}=\frac{\ker d\cap A^{1,1}}{\del\delbar A^{0,0}_s}.
\]
In general, there seems no reason to think this injection is also a surjection. Note that $j$ being surjective would imply $h^{1,1}_{BC}$ is an almost complex invariant on $4$-manifolds.

We can re-obtain the previous injection of $(1,1)$-forms as a particular case of the following observation. Let $(M,J,g,\omega)$ be a compact almost Hermitian manifold of real dimension $2n$ and intersect the Bott-Chern decomposition with the space $\ker d\cap A^{p,q}_s$, deriving
\[
\ker d\cap A^{p,q}_s=\H^{p,q}_{BC}\cap A^{p,q}_s\oplus\ker d\cap A^{p,q}_s\cap\del\delbar A^{p-1,q-1}.
\]
Therefore, there is a well defined injection
\[
\H^{p,q}_{BC}\cap A^{p,q}_s\overset{j}{\longrightarrow} H^{p,q}_{BC}.
\]
Summing up, we have
\begin{proposition}\label{BC-cohomology}
Let $(M,J,g,\omega)$ be a compact almost Hermitian manifold of real dimension $2n$. Then we have an injection
\[
\H^{p,q}_{BC}\cap A^{p,q}_s\overset{j}{\longrightarrow} H^{p,q}_{BC}.
\]
\end{proposition}


\begin{thebibliography}{20}
\bibitem{ADT} D. Angella, G. Dloussky, A. Tomassini, On Bott-Chern cohomology of compact complex surfaces, {\em Ann. Mat. Pura Appl. (4)} {\bf 195} (2016), 199--217.
\bibitem{BPV} W. Barth, C. Peters, A. Van de Ven, {\em Compact Complex Surfaces}, Springer-Verlag, 1984.
\bibitem{CZ} H. Chen, W. Zhang, Kodaira Dimensions of Almost Complex Manifolds I, {\tt arXiv:1808.00885}, 2018.
\bibitem{CW} J. Cirici, S.O. Wilson, Dolbeault cohomology for almost complex manifolds, {\tt arXiv:1809.01416}, 2018, to appear in {\em Adv. in Math.}.
\bibitem{CPS} R. Coelho, G. Placini, J. Stelzig, Maximally non-integrable almost complex structures: an h-principle and cohomological properties, {\tt arXiv:2105.12113}, 2021.
\bibitem{DK} S.K. Donaldson, P.B. Kronheimer, {\em The Geometry of Four-Manifolds}, Clarendon Press, Oxford, 1990. 
\bibitem{E}L.C. Evans, {\em Partial differential equations}, Graduate Studies in Mathematics, 19 (2nd ed.), American Mathematical Society, Providence, RI, 2010.
\bibitem{Ga} P. Gauduchon, Le théorème de l'excentricité nulle. (French) {\em C. R. Acad. Sci. Paris Sér. A-B} {\bf 285} (1977), 387-390.
\bibitem{Ha} K. Hasegawa, Complex and Kähler structures on compact solvmanifolds, \emph{J. Symplectic Geom.} {\bf 3} (2005), 749–767.
\bibitem{Hi} F. Hirzebruch, Some problems on differentiable and complex manifolds, \emph{Ann. Math. (2)} \textbf{60}, (1954). 213--236.
\bibitem{Ho} T. Holt, Bott-Chern and $\delbar$-Harmonic forms on Almost Hermitian 4-manifolds, {\tt arXiv:2111.00518}, 2021.
\bibitem{HZ} T. Holt, W. Zhang, Harmonic Forms on the Kodaira-Thurston Manifold, {\tt arXiv:2001.10962}, 2020.
\bibitem{HZ2} T. Holt, W. Zhang, Almost Kähler Kodaira-Spencer problem, {\tt arXiv:2010.12545}, 2021, to appear in {\em Math. Res. Lett.}.
\bibitem{H} D. Huybrechts, {\em Complex Geometry. An Introduction}, Springer, Berlin, 2005. 
\bibitem{F} K. Kodaira, {\em Complex Manifolds and Deformation of Complex Structures}, Appendix Elliptic Partial Differential Operators on a Manifold by Daisuke Fujiwara, Springer, 1986.
\bibitem{KS} K. Kodaira, D.C. Spencer, On deformations of complex analytic structures, III. Stability theorems for complex structures, {\em Ann. Math. (2)} {\bf 71} (1960), 43--76.
\bibitem{S} M. Schweitzer, Autour de la cohomologie de Bott-Chern, {\tt arXiv:0709.3528v1}, 2007.
\bibitem{TT0} N. Tardini, A. Tomassini, Differential operators on almost-Hermitian manifolds and harmonic forms, \emph{Complex Manifolds} {\bf 7} (2020), 106–128.
\bibitem{TT} N. Tardini, A. Tomassini, $\delbar$-harmonic forms on 4-dimensional almost-hermitian manifolds, {\tt arXiv:2104.10594}, 2021, to appear in {\em Math. Res. Lett.}.






\end{thebibliography}
\end{document}